\documentclass[11pt,leqno,oneside]{amsart}
\usepackage[colorlinks, linkcolor=teal, citecolor=violet]{hyperref}
\usepackage{amsmath,amsfonts,amssymb,amsthm,enumerate,mathtools,comment}
\usepackage{color,mathscinet,doi}
\usepackage[a-2u]{pdfx}
\usepackage[mathcal]{euscript}
\usepackage[a4paper, left=2cm, right=2cm, top=3cm, bottom=3cm]{geometry}
\usepackage[numbers,sort&compress]{natbib}
\usepackage{caption,subcaption}
\usepackage{tikz,pgfplots}
\usetikzlibrary{patterns}
\pgfplotsset{compat=1.15}
\newcommand{\R}{\mathbb{R}}
\newcommand{\rn}{\mathbb{R}^n}
\newcommand{\N}{\mathbb{N}}
\renewcommand{\d}{{\fam0 d}}
\newcommand{\sprt}{\operatorname{supp}}
\newcommand{\esssup}{\operatornamewithlimits{ess\,sup}}
\newcommand{\essinf}{\operatornamewithlimits{ess\,inf}}
\newcommand{\nk}{\frac nk}
\newcommand{\kn}{\frac kn}
\newcommand{\ir}{\frac{1}{r}}
\newcommand{\ip}{\frac{1}{p}}
\newcommand{\rpar}{\varrho}
\newcommand{\multiindex}{\beta}
\newcommand{\bmc}{\alpha}
\let\tilde\widetilde
\DeclarePairedDelimiter\abs{\lvert}{\rvert}
\DeclarePairedDelimiter\nrm{\lVert}{\rVert}

\usepackage{xspace}
\makeatletter
\DeclareRobustCommand\onedot{\futurelet\@let@token\@onedot}
\def\@onedot{\ifx\@let@token.\else.\null\fi\xspace}
\def\eg{e.g\onedot}
\def\ie{i.e\onedot}
\def\cf{cf\onedot}
\makeatother

\newtheoremstyle{MyPlain}{}{}{\itshape}{}{\bfseries}{.}{5pt plus 4pt minus 3pt}{\thmname{#1}\thmnumber{ #2}\thmnote{ \textbf{[#3]}}}
\theoremstyle{MyPlain}
\newtheorem{theorem}{Theorem}[section]

\newtheorem{proposition}[theorem]{Proposition}

\newtheorem{question}{Question}
\newtheoremstyle{MyRemark}{}{}{\upshape}{}{\bfseries}{.}{5pt plus 1pt minus 1pt}{}
\theoremstyle{MyRemark}
\newtheorem*{definition}{Definition}
\newtheorem*{example}{Example}
\newtheorem{remark}[theorem]{Remark}

\numberwithin{equation}{section}

\expandafter\let\expandafter\oldproof\csname\string\proof\endcsname
\let\oldendproof\endproof
\renewenvironment{proof}[1][\proofname]{%
  \oldproof[{{\bf #1.}}]%
}{\oldendproof}

\makeatletter
\def\paragraph{\bigskip\@startsection{paragraph}{4}%
  \z@\z@{-\fontdimen2\font}%
  {\normalfont\bfseries}}
\makeatother

\begin{document}

\title{Maximal non-compactness of Sobolev embeddings}

\begin{abstract}
It has been known that sharp Sobolev embeddings into weak Lebesgue spaces are
non-compact but the question of whether the measure of non-compactness of such an
embedding equals to its operator norm constituted a well-known open problem.
The existing theory suggested an argument that would possibly solve the problem
should the target norms be disjointly superadditive, but the question of
disjoint superadditivity of spaces $L^{p,\infty}$ has been open, too. In this
paper, we solve both these problems. We first show that weak Lebesgue spaces are
never disjointly superadditive, so the suggested technique is ruled out. But
then we show that, perhaps somewhat surprisingly, the measure of non-compactness
of a sharp Sobolev embedding coincides with the embedding norm nevertheless, at
least as long as $p<\infty$. Finally, we show that if the target space is
$L^{\infty}$ (which formally is also a weak Lebesgue space with $p=\infty$),
then the things are essentially different. To give a comprehensive answer
including this case, too, we develop a new method based on a rather
unexpected combinatorial argument and prove thereby a general
principle, whose special case implies that the measure of non-compactness, in
this case, is strictly less than its norm. We develop a technique that
enables us to evaluate this measure of non-compactness exactly.
\end{abstract}

\author{Jan Lang\textsuperscript{1}}
\address{\textsuperscript{1}Department of Mathematics,
The Ohio State University,
231 West 18th Avenue,
Columbus, OH 43210-1174}

\author{V\'\i t Musil\textsuperscript{2,3}}
\address{\textsuperscript{2}Dipartimento di Matematica e Informatica ``Ulisse Dini'',
University of Florence,
Viale Morgagni 67/A, 50134
Firenze,
Italy}
\address{\textsuperscript{3}Institute of Mathematics,
Czech Academy of Sciences,
\v Zitn\' a 25,
115~67, Prague~1,
Czech Republic}

\author{Miroslav Ol\v s\'ak\textsuperscript{4}}
\address{\textsuperscript{4}Department of Computer Science,
Technikerstrasse 21a,
A-6020 Innsbruck,
Austria}

\author{Lubo\v s Pick\textsuperscript{5}}
\address{\textsuperscript{5}Department of Mathematical Analysis,
Faculty of Mathematics and Physics,
Charles University,
So\-ko\-lo\-vsk\'a~83,
186~75 Praha~8,
Czech Republic}

\urladdr{
	0000-0003-1582-7273 {\rm(Lang)},
	0000-0002-9361-1921 {\rm(Ol\v s\'ak)},
	0000-0001-6083-227X {\rm(Musil)},\hfill\break
	0000-0002-3584-1454 {\rm(Pick)}
}
\email{pick@karlin.mff.cuni.cz}


\maketitle

\bibliographystyle{abbrv_doi}

\section*{How to cite this paper}
\noindent
This paper has been accepted for publication in \emph{The Journal of Geometric
Analysis} and is available~on
\begin{center}
	\url{https://doi.org/10.1007/s12220-020-00522-y}.
\end{center}
Should you wish to cite this paper, the authors would like to cordially ask you
to cite it appropriately.

\section{Introduction}

\noindent
Given a linear mapping acting between two (quasi)normed linear spaces, one of
the most important questions is whether it is compact. Compactness is often
desired or even indispensable for specific applications in different areas of
mathematics. It plays an important role in theoretical parts of functional
analysis such as, for instance, in the proof of the Schauder fixed point
theorem, and also in the most customary applications of functional analysis
such as proving existence, uniqueness, and regularity of solutions to partial
differential equations via compact embeddings of Sobolev-type spaces into
various other function spaces. However, more often than not, the mapping in
question is not~compact.

For a non-compact mapping, more subtle techniques have to be developed. For
example, in 1955, G. Darbo \citep{Dar:55} extended the Schauder theorem to
certain types of non-compact operators. The main tool she used was the measure
of non-compactness, which had been introduced earlier by Kuratowski
in~\citep{Kur:30} in connection with different problems in general topology. An
important generalization was later added by Sadovskii \citep{Sad:68}. Later
still, the notion of the measure of non-compactness proved to be very useful
for various applications, and a formidable theory was developed. See
\eg~\citep{Tri:78,Pie:07} for a general survey and more references.

The concept of measure of non-compactness of a mapping is a~good device for
quantifying \emph{how bad} the non-compactness is, or, perhaps, \emph{how
far} from the class of compact maps the given operator lies. Let us recall
its definition here.
Throughout the paper, $B_X$ denotes the open unit ball in $X$ centered at the origin.
\begin{definition}
Let $X$ and $Y$ be (quasi)normed linear spaces and let $T$ be a~bounded mapping
defined on $X$ and taking values in $Y$, a fact we will denote by $T\colon X\to
Y$. The \emph{ball measure of non-compactness} $\bmc(T)$ of $T$ is defined
as the infimum of radii $\rpar>0$ for which there exists a finite set
of balls in $Y$ of radii $\rpar$ that covers $T(B_X)$.
\end{definition}

There are other examples of measures of non-compactness, for instance, the
Kuratowski measure of non-compactness, which is defined analogously but with
balls replaced by arbitrary sets of diameter not exceeding $\rpar$. In our
analysis, we focus on the ball measure of non-compactness even if we sometimes
avoid the adjective ``ball''.

The measure of non-compactness is an important geometric feature of images of
bounded sets under an operator, see \eg~\citep{BG:80}. It is intimately
connected for instance to the classical entropy numbers or certain types of the
so-called $s$-numbers, see \eg~\citep{EdmTri:96,Pie:07,EdmEv:18}. Its
importance stems among other reasons from Carl's inequality
\citep{CarTri:80,Car:81} which establishes its relationship to eigenvalues of an
operator. The measure of non-compactness is also related to the essential spectrum
of a bounded map \citep{EdmEv:18}.

From the definition of the measure of non-compactness, we easily observe that
\begin{equation*}
    0 \le \bmc(T) \le \nrm{T},
\end{equation*}
where $\nrm{T}$ denotes the norm of the operator $T$ considered as a map from
$X$ to $Y$. Then $T$ is compact if and only if $\bmc(T)=0$, and it is as
non-compact as possible if $\bmc(T)=\nrm{T}$. In the latter case, we say
that $T$ is \emph{maximally non-compact}.

\begin{example}
A simple example of a maximally non-compact operator is the embedding of
sequence spaces
\begin{equation*}
	I\colon\ell^p\to\ell^q
		\quad\text{for $1\le p\le q<\infty$,}
\end{equation*}
where $I$ is the identity (or the \emph{embedding operator}).  Indeed, we
obviously have $\nrm{I}=1$. Suppose thus that $\bmc(I)<1$ and fix
$\rpar$ such that $\bmc(I)<\rpar<1$. Then there is some $m\in\N$ and elements
$y^{1},\dots,y^{m}$ of $\ell^{q}$ such that
\begin{equation*}
	B_{\ell^{p}}\subset\bigcup\nolimits_{k=1}^{m}
		\bigl(y^{k}+\rpar B_{\ell^{q}}\bigr).
\end{equation*}
Since
all $y^{k}$'s belong to $\ell^{q}$, they are also elements of $c_0$. Hence
there is a $j\in\N$ such that $(y^{k})_j<1-\rpar$ for each $k=1,\dots,m$. But
then the vector $e^j=(0,\dots,0,1,0,\dots)$ with one on the $j$-th position
belongs to $B_{\ell^p}$ and does not belong to any of the balls $y^k+\rpar B_{\ell^q}$,
which is a contradiction. Consequently, $\bmc(I)=1$ and $I$ is
maximally non-compact.

Interestingly, the situation is dramatically different when the target space is
$\ell^{\infty}$. Then, for any fixed $p\in[1,\infty)$, the norm of
$I\colon\ell^p\to\ell^{\infty}$ is again equal to 1, but $I$ is no longer maximally
non-compact. More precisely, we will demonstrate that $\bmc(I) \le 2^{-1/p}$.
To this end, denote $\sigma=2^{1-1/p}$, fix $\rpar>\sigma/2$, and consider
$m\in\N$ such that
\begin{equation} \label{E:m-def}
	\Bigl( 1+\frac{1}{m} \Bigr)\frac{\sigma}{2} < \rpar.
\end{equation}
Define $\lambda_k = \tfrac{\sigma k}{2m}$ for $k=-m,\dots,m$
and let $y^k$ be the constant sequence defined by $(y^{k})_j=\lambda_k$ for
every $j\in\N$ and $k=-m,\dots,m$.
We show that
\begin{equation} \label{E:Blp}
	B_{\ell^p} \subset \bigcup\nolimits_{k=-m}^{m}
		\bigl( y^k + \rpar B_{\ell^\infty} \bigr),
\end{equation}
proving $\bmc(I) \le \rpar$. Assume that $y\in B_{\ell^{p}}$. Then $y\in B_{\ell^\infty}$
and $\abs{y_j}\le 1 \le \sigma$ for every $j\in\N$.
We claim that
\begin{equation} \label{E:sup-inf-example}
	\sup y - \inf y\le\sigma.
\end{equation}
Indeed,
given $\varepsilon>0$, we find $s,i\in\N$ such that $y_s>\sup y-\varepsilon$
and $y_i<\inf y+\varepsilon$. We get
\begin{equation*}
	1 \ge \nrm{y}_{\ell^{p}}
		\ge \left(\abs{y_s}^p+\abs{y_i}^p\right)^{\frac{1}{p}}
		\ge 2^{\frac{1}{p}-1}(\abs{y_s}+\abs{y_i})
		> \frac{1}{\sigma}(\sup y-\inf y - 2\varepsilon),
\end{equation*}
and the claim follows on sending $\varepsilon\to 0_+$.
Now, if $\inf y = -\sigma$, then $y_j\in[\sigma,0]$ for
each $j\in\N$ and $y\in y^{-m}+\rpar B_{\ell^\infty}$.
If $\inf y\in (-\sigma,0]$ instead, then there is a unique $k\in\{-m+1,\ldots,m\}$
such that $\inf y + \sigma/2\in(\lambda_{k-1},\lambda_k]$.
Then, by the choice of $\rpar$ and inequality \eqref{E:sup-inf-example},
\begin{equation*}
	\lambda_k + \rpar
		> \lambda_k + \frac{\sigma}{2}
		\ge \inf y + \sigma
		\ge \sup y.
\end{equation*}
On the other hand, using the definition of $\lambda_{k}$ and \eqref{E:m-def},
\begin{equation*}
	\inf y
		> \lambda_{k-1} - \frac{\sigma}{2}
		= \lambda_k - \frac{\sigma}{2}\Bigl( 1+\frac{1}{m} \Bigr)
		> \lambda_k - \rpar.
\end{equation*}
Altogether, $y\in y^k+\rpar B_{\ell^{\infty}}$, and \eqref{E:Blp} follows.

We showed that $\bmc(I)\le\sigma/2=2^{-1/p}<1$,
whence $I$ is not maximally non-compact. It is not difficult
to prove that $\bmc(I)$ is actually equal to $2^{-1/p}$
(using similar reasoning as in the proof of~Theorem~\ref{T:beta-of-emb-limiting}), but that is beside the
point here. This example is simple (and likely to be known, although we did
not find it in the literature in this exact form),
but it is a good illustration of much more involved attractions below.
\end{example}

An important operator to which the theory is often applied is the identity
acting from a Sobolev space into another function space. Such an identity is also called
a~\emph{Sobolev embedding}.
Let $n\in\N$ and let $\Omega$ be an open, bounded and
nonempty set in $\rn$.
Recall that the Sobolev space $V^k_0X(\Omega)$ for $k\in\N$ is defined
as a collection of all measurable functions $u\colon\Omega\to\R$ whose extension by zero
outside $\Omega$ is $k$-times weakly differentiable and $\abs{\nabla^k u}\in X(\Omega)$.
Here $\nabla^k u = (D^\multiindex u)_{\abs{\multiindex}=k}$ is the vector
of all the derivatives of $u$ of order $k$, where, for an $n$-dimensional multiindex $\multiindex$, $D^\multiindex$ denotes
$\partial^\multiindex/\partial x^\multiindex$.
Once equipped with the norm
\begin{equation*}
	\nrm{u}_{V^{k}_0X(\Omega)}
		= \sum_{\abs{\multiindex}=k} \nrm{D^\multiindex u}_{X(\Omega)},
\end{equation*}
the Sobolev space $V^k_0X(\Omega)$ forms a Banach space.
If $X$ represents a classical Lebesgue space $L^p$
for some $p\in[1,\infty]$, we simply write $V^{k,p}_0(\Omega)$.

The compactness of a~Sobolev embedding can constitute a crucial step in many
applications in partial differential equations, probability theory, calculus of
variations, mathematical physics and other disciplines and therefore it has
been widely studied alongside with quantification of its absence when
appropriate. To name just one of many interesting connections, let us recall
that spectral properties of the Laplacian are governed by the measure of
non-compactness of a~Sobolev embedding~\citep{EdmEv:18,EvHar:87}.

The variational approach to partial differential equations with singular
coefficients often requires the use of an embedding of a Sobolev space into a
two-parameter Lorentz space. Compactness properties of such embeddings are
crucial under various circumstances
\citep{Lio1:84,Lio2:84,Lio1:85,Lio2:85,Str:84,Sol:95,JaSo:99}. Of particular
interest is an embedding into a Lorentz space whose second index is equal to infinity, equivalent to a weak Lebesgue space, see~\citep{CaRuTa:13}.
However, it is a rule of thumb that if a~Sobolev embedding is sharp in the
sense of function spaces, then it is never
compact~\citep{Sol:95,T4,Sla:15,CaMi:19}. It is thus of interest to study how
bad is its non-compactness.

\medskip
Our main goal in this paper is to study maximal non-compactness of general
Sobolev embeddings with emphasis on embeddings involving Lorentz--Sobolev
spaces. It is worth noticing that classical Sobolev embeddings built upon
Lebesgue spaces are included as particular instances.

The classical Sobolev embedding theorem (\cf \eg~\citep{Ada:75, AF:03, Maz:85})
asserts that if $n,k\in\N$, $k<n$, $\Omega$ is an open set in $\rn$,
$p\in[1,\nk)$ and $p^*=\frac{np}{n-kp}$, then one has
\begin{equation}\label{E:classical-sobolev}
    I\colon V^{k,p}_0(\Omega)\to L^{p^*}(\Omega),
\end{equation}
where $I$ is the identity operator. The Sobolev space on the left-hand side,
$V_0^{k,p}(\Omega)$, consists of functions of highest regularity ($k$) and
relatively small integrability ($p$) while, on the right-most side, the degree
of integrability is increased (note that $p^*>p$), balancing the loss of
regularity. This example explains the great importance of Sobolev embeddings:
it is an appropriate tool for ``trading regularity for integrability''.

Like the Lebesgue $L^p$ norm, the two-parameter Lorentz $L^{p,q}$ norm (or
quasinorm, in general) captures the $p$-integrability of a function whereas the
parameter $q$ measures how ``spread out'' the mass of the function is. This
extra index $q$ thus provides a fine-tuning of the $L^p$ space with the
property that $L^{p,q_1}\subseteq L^{p,q_2}$ if $q_1\le q_2$.  Setting $q=p$,
we recover the Lebesgue space $L^{p,p}=L^p$ to reveal that the Lorentz spaces
refine the Lebesgue scale. For these and other details concerning Lorentz
spaces, see \eg~\citep{BS:88,Pic:13}.

Having Lorentz spaces at hand, the classical Sobolev embedding
\eqref{E:classical-sobolev} can be enhanced to
\begin{equation}\label{E:classical-sobolev-lorentz}
    I\colon V^{k,p}_0(\Omega)\to L^{p^*,p}(\Omega),
\end{equation}
in which the Lebesgue target space $L^{p^*}(\Omega)$ is replaced with the
essentially smaller Lorentz space $L^{p^*,p}(\Omega)$ making thus the embedding
stronger.
The latter embedding is known to be
optimal in the sense that the target space cannot be replaced by any smaller
rearrangement-invariant Banach function space~\citep{EKP,T2}. However, none of
the embeddings~\eqref{E:classical-sobolev}, \eqref{E:classical-sobolev-lorentz}
is compact, and, as recent advances show \citep{T4,Sla:15}, even when the target space is
replaced by a considerably larger space $L^{p^*,\infty}(\Omega)$, the resulting embedding
\begin{equation} \label{E:limiting-embedding}
    I\colon V^{k,p}_0(\Omega)\to L^{p^*,\infty}(\Omega)
\end{equation}
is still not compact. In~\citep{Hen:03} it is shown that the Sobolev
embedding~\eqref{E:classical-sobolev} is maximally non-compact.
In~\citep{Bou:19}, this result is extended to more general Sobolev embeddings
of the form
\begin{equation*}
	I\colon V^{k}_0L^{p,q}(\Omega)\to L^{p^*,s}(\Omega)
		\quad \text{for every $1\leq q\leq s<\infty$}.
\end{equation*}
These results leave open the case when $s=\infty$.
Thus a natural question arises.

\begin{question} \label{Q:1}
Given $k,n\in\N$, $k<n$, and $p\in[1,\tfrac nk)$,
is~\eqref{E:limiting-embedding} maximally non-compact?
\end{question}
The key feature of the
approach of both~\citep{Hen:03} and~\citep{Bou:19} is the fact that any Lebesgue
space $L^p(\Omega)$ for $p\in[1,\infty)$, as well as any Lorentz space
$L^{p,q}(\Omega)$ for $p,q\in[1,\infty)$, is disjointly superadditive.

\begin{definition}
We say that a (quasi)normed linear space $X(\Omega)$ containing
functions defined on $\Omega$ is \emph{disjointly superadditive} if there
exist $\gamma>0$ and $C>0$ such that for every $m\in\N$ and
every finite sequence of functions $\{f_k\}_{k=1}^m$ with pairwise disjoint supports in
$\Omega$ one has
\begin{equation*}
	\sum\nolimits_{k=1}^{m} \nrm{f_k}_{X(\Omega)}^\gamma
		\le C \nrm*{\sum\nolimits_{k=1}^{m} f_k}_{X(\Omega)}^\gamma.
\end{equation*}
\end{definition}
In order to answer Question~\ref{Q:1}, one should first investigate the
following closely related mystery.

\begin{question} \label{Q:2}
Is the space $L^{p^*,\infty}(\Omega)$ disjointly superadditive?
\end{question}

The reason for considering Question~\ref{Q:2} is that if the answer to it was
positive, then it would be very likely that using some not-so-difficult
modification of techniques of~\citep{Hen:03} and~\citep{Bou:19} one should be
able to prove that the answer to Question~\ref{Q:1} is affirmative, too.
However, it turns out
(Theorem~\ref{T:weak-lebesgue-non-disjointly-superadditive}) that the answer to
Question~\ref{Q:2} is negative, that is, the space $L^{p^*,\infty}(\Omega)$ is
not disjointly superadditive.

This result is interesting on its own, but it leaves us shorthanded as far as
Question~\ref{Q:1} is concerned. So, in order to answer it we have to develop
new techniques.  In Section~\ref{S:maximal-non-compactness} we prove that the
answer to Question~\ref{Q:1} is positive, even though the methods
of~\citep{Hen:03} and~\citep{Bou:19} do not apply. More precisely, we in fact
show that a slightly more general embedding than~\eqref{E:limiting-embedding},
namely
\begin{equation*}
	I\colon V^{k}_0L^{p,q}(\Omega)\to L^{p^*,\infty}(\Omega),
\end{equation*}
is maximally non-compact for every $k,n\in\N$, $k<n$, $p\in[1,\nk)$ and
$q\in[1,\infty]$. We obtain this result (Theorem~\ref{T:max-non-compact}) as
a~consequence of a~fairly comprehensive principle
(Theorem~\ref{T:max-non-compact-general}) which postulates maximal
non-compactness of embeddings into weak Lebesgue spaces provided that the
underlying identity operator has certain shrinking property, which roughly
states that its norm over an open set $\Omega$ is attained at functions having
their supports restricted to any (arbitrarily small) open subset of $\Omega$.
In order to be able to apply this theory to our purposes, we need to know that
the identity operator in~\eqref{E:limiting-embedding} has shrinking property.
We establish this fact in Proposition~\ref{P:shrinking-sobolev}.

The techniques of Section~\ref{S:maximal-non-compactness} do not work for the
case when the target space is $L^{\infty}(\Omega)$ although it is also a
weak Lebesgue space. It was shown in~\citep{St:81} that a proper domain partner
for a Sobolev embedding into $L^{\infty}(\Omega)$ is the Lorentz space
$L^{\nk,1}(\Omega)$. More precisely, we have
\begin{equation} \label{E:emb-limiting}
    I\colon V^{k}_0L^{\nk,1}(\Omega)\to L^{\infty}(\Omega).
\end{equation}
There is thus one more natural, and still unanswered, problem.

\begin{question} \label{Q:3}
Given $k,n\in\N$, $k\le n$, is the Sobolev embedding~\eqref{E:emb-limiting}
maximally non-compact?
\end{question}

Let us first recall that in the very special (one-dimensional) case when
$n=k=1$, the answer is known. In the one-dimensional setting, $\Omega$ is
replaced by a compact interval $[a,b]$ and $L^{\infty}(\Omega)$ by $C(a,b)$,
the space of all continuous functions on $[a,b]$ endowed with the
$L^{\infty}$-norm. It is shown in~\citep{Hen:03} and~\citep{Bou:19}
that~\eqref{E:emb-limiting} is not maximally non-compact in this case. However,
when $n=k=1$, the Lorentz space $L^{\nk,1}(\Omega)$ on the domain position
collapses to the Lebesgue space $L^{1}(\Omega)$ and things get much simpler.

The question of extending this result to the higher-dimensional and
higher-order case has been one of the notoriously difficult open problems in
the theory. In Section~\ref{S:embeddings-into-l-infty}, we solve this problem
with the help of a new method which we develop for this purpose. At the end, we
show that the answer to Question~\ref{Q:3} is negative
(Theorem~\ref{T:beta-of-emb-limiting}), and we obtain this fact as a
consequence of a generic quantitative statement
(Theorem~\ref{T:beta-lower-of-X-to-ell-infty}) which works for a wide variety
of operators.

The results contained in Theorems~\ref{T:beta-lower-of-X-to-ell-infty}
and~\ref{T:beta-of-emb-limiting} are interesting for at least two reasons.
First, it is striking that the situation is so drastically different from any
other embedding into a~weak Lebesgue space though it accords with the example
mentioned in the introduction. The second is the innovative method of proof of
Theorem~\ref{T:beta-lower-of-X-to-ell-infty} based on a combinatorial argument
involving a~coloring-type problem, see Figure~\ref{fig:coloring}.  Such line of
argumentation is rarely seen in the area of mathematical analysis.

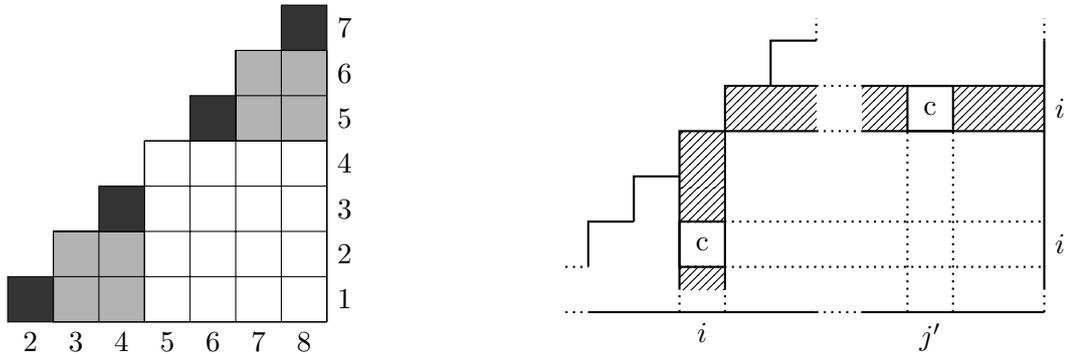
\begin{figure}[ht]
\centering
  \begin{subfigure}[t]{0.46\linewidth}
		\centering
		\begin{tikzpicture}[scale=0.6]
			\draw (3,0) grid + (4,4) rectangle (3,0);
			\foreach\x in {0,4}
				\draw[fill=black!30] (\x+1,\x) grid + (2,2) rectangle (\x+1,\x);
			\foreach\x in {0,2,4,6}
				\draw[fill=black!80] (\x,\x) rectangle +(1,1);
			\foreach\x in {2,...,8}
				\node[below] at (\x-1.5,0) {\x};
			\foreach\x in {1,...,7}
				\node[right] at (7,\x-0.5) {\x};
		\end{tikzpicture}
		\caption{An example of a valid coloring with $m$ colors. Triangle of side
		length $2^m-1$ consists of two smaller triangles of side length $2^{m-1}-1$
		and of a square of length $2^{m-1}$. Pick a new color for the square and
		proceed inductively on smaller triangles.}
		\label{fig:coloring:sufficiency}
	\end{subfigure}
	\hfil
  \begin{subfigure}[t]{0.46\linewidth}
		\centering
		\begin{tikzpicture}[scale=0.6, thick]
			\draw[dotted] (1.5,1) -- (2,1);
			\draw (2,1) -- (7,1);
			\draw[dotted] (7,1) -- (8,1);
			\draw (8,1) -- (12,1);
			\draw[dotted] (12,1) -- (12,1.5);
			\draw (12,1.5) -- (12,7);
			\draw[dotted] (12,7) -- (12,7.5);
			\foreach\x in {2,...,6}
					\draw (\x,\x)--(\x,\x+1)--(\x+1,\x+1);
			\draw[dotted] (1.5,2) -- (2,2);
			\draw[dotted] (7,7) -- (7,7.5);
			\draw (4,2) rectangle + (1,1) node[midway] {c};
			\draw (9,5) rectangle + (1,1) node[midway] {c};
			\fill[pattern=north east lines] (4,3) rectangle + (1,2);
			\fill[pattern=north east lines] (4,1.5) rectangle + (1,0.5);
			\fill[pattern=north east lines] (5,5) rectangle + (2,1);
			\fill[pattern=north east lines] (8,5) rectangle + (1,1);
			\fill[pattern=north east lines] (10,5) rectangle + (2,1);
			\node[right] at (12,5.5) {$i$};
			\draw[dotted] (4,1) -- (4,1.5);
			\draw[dotted] (5,1) -- (5,1.5);
			\node[below] at (4.5,1) {$i$};
			\draw (4,1.5) -- (4,4);
			\draw (5,1.5) -- (5,5) -- (7,5);
			\draw[dotted] (7,5) -- (8,5);
			\draw (8,5) -- (12,5);
			\draw (6,6) -- (7,6);
			\draw[dotted] (7,6) -- (8,6);
			\draw (8,6) -- (12,6);

			\draw[dotted] (4,2) -- (12,2);
			\draw[dotted] (4,3) -- (12,3);
			\node[right] at (12,2.5) {$i$};
			\draw[dotted] (9,1) -- (9,5);
			\draw[dotted] (10,1) -- (10,5);
			\node[below] at (9.5,1) {$j'$};
		\end{tikzpicture}
		\caption{Sets of colors of $i$-th and $j$-th row ($C_i$ and $C_j$, resp.)
		differ. The color $c\in C_i$ of the element $w_{i,j}$ cannot appear in $C_j$,
		since otherwise $j$-th row and $j$-th column would share this color.
		To achieve $2^m-1$ distinct (nonempty) sets, at least $m$ are needed.}
		\label{fig:coloring:necessity}
	\end{subfigure}
\caption{Triangle coloring problem: Color a grid triangle of side length
$2^m-1$ provided that $j$-th line and $j$-th column do not share a color for
every $j=1,\dots,2^{m}-1$. Show that at least $m$ colors are needed.}
\label{fig:coloring}
	\end{figure}

\section{Lack of disjoint superadditivity of weak Lebesgue spaces}

This section is devoted to the analysis of Question~\ref{Q:2}.
We recall some definitions and fix the notation first.

\paragraph{Rearrangements}
For a~measurable function $u\colon\Omega\to\R$, its nonincreasing
rearrangement, $u^*\colon[0,\infty)\to[0,\infty]$, is defined by
\begin{equation*}
	u^*(t) = \inf\{\lambda>0:\abs{\{x\in\Omega: \abs{u(x)}>\lambda\}}\leq t\}
		\quad\text{for $t\in[0,\infty)$}.
\end{equation*}
The absolute values $|\cdot|$ denote the Lebesgue measure.
The {maximal nonincreasing rearrangement} of $u$, namely the function
$u^{**}\colon(0,\infty)\to[0,\infty]$, is defined by
\begin{equation*}
	u^{**}(t) = \frac{1}{t}\int_0^tu^*(s)\,\d s
		\quad\text{for $t\in(0,\infty)$}.
\end{equation*}
An alternative formula for $u^{**}$ reads as
\begin{equation} \label{E:maximal-alt}
	u^{**}(t) = \frac{1}{t} \sup \int_E \abs{u(x)}\,\d x
		\quad\text{for $t\in(0,\abs\Omega]$,}
\end{equation}
where the supremum is taken over all measurable sets $E\subseteq\Omega$ such that
$\abs{E}=t$, see \eg~\citep[Proposition~7.4.5]{Pic:13}.
Note that the maximal nonincreasing rearrangement is subadditive, that is, given measurable functions $u,v\colon\Omega\to\R$, we have
\begin{equation} \label{E:f**-subadditivity}
	(u+v)^{**}(t) \le u^{**}(t) + v^{**}(t)
		\quad\text{for $t\in(0,\infty)$},
\end{equation}
while for the nonincreasing rearrangement we have only
\begin{equation} \label{E:f*-subadditivity}
	(u+v)^{*}(s+t) \le u^{*}(s) + v^{*}(t)
		\quad\text{for $s,t\in(0,\infty)$}.
\end{equation}

\paragraph{Lorentz spaces}
Given $0<p,q\le\infty$, the functional
$\|\cdot\|_{L^{p,q}(\Omega)}$ is defined by
\begin{equation} \label{E:Lpq-def}
	\nrm{u}_{L^{p,q}(\Omega)}
		= \nrm*{s^{\frac{1}{p}-\frac{1}{q}}u^*(s)}_{L^q(0,\abs{\Omega})}
\end{equation}
for a~measurable function $u\colon\Omega\to\R$.
We adopt the notation that $1/\infty=0$.
If either $1<p<\infty$ and
$1\leq q\leq\infty$, or $p=q=1$, or $p=q=\infty$, then
$\|\cdot\|_{L^{p,q}(\Omega)}$ is equivalent to a~norm (\cf \eg~\citep{BS:88}
for details), in other cases it is a~quasinorm. We further define the
functional $\|\cdot\|_{L^{(p,q)}(\Omega)}$ as
\begin{equation} \label{E:L(pq)-def}
	\nrm{u}_{L^{(p,q)}(\Omega)}
		= \nrm*{s^{\frac{1}{p}-\frac{1}{q}}u^{**}(s)}_{L^q(0,\abs{\Omega})}
\end{equation}
for a~measurable function $u\colon\Omega\to\R$.  If either $0< p<\infty$ and
$1\leq q\leq\infty$, or $p=q=\infty$, then $\|\cdot\|_{L^{(p,q)}(\Omega)}$ is
a~norm, in other cases it is a~quasinorm. The functionals
$\|\cdot\|_{L^{p,q}(\Omega)}$ and  $\|\cdot\|_{L^{(p,q)}(\Omega)}$ are
called {Lorentz (quasi)norms}, and the corresponding
spaces $L^{p,q}(\Omega)$ and $L^{(p,q)}(\Omega)$, defined as collections of
all measurable functions $u\colon\Omega\to\R$ such that
$\nrm{u}_{L^{p,q}(\Omega)}<\infty$ or $\nrm{u}_{L^{(p,q)}(\Omega)}<\infty$,
respectively, are called {Lorentz spaces}.
Besides the relations mentioned in the introduction, it always holds that
$L^{(p,q)}(\Omega)\to L^{p,q}(\Omega)$ and
if either $p\in(1,\infty)$ and $q\in[1,\infty]$, or $p=q=\infty$,
then $L^{(p,q)}(\Omega)=L^{p,q}(\Omega)$.
By equality of spaces we mean the equality of the sets and equivalence of their norms.

\paragraph{Weak Lebesgue spaces}
Recall that weak Lebesgue spaces coincide with Lorentz spaces when the second
index equals to infinity, namely with $L^{p,\infty}(\Omega)$ and
$L^{(p,\infty)}(\Omega)$ and their norms are given by
\begin{equation} \label{E:weak-lebesgue-norms}
	\nrm{u}_{L^{p,\infty}(\Omega)}
		= \sup_{t\in(0,\abs{\Omega})} t^{\ip}\,u^*(t)
	\quad\text{and}\quad
	\nrm{u}_{L^{(p,\infty)}(\Omega)}
		= \sup_{t\in(0,\abs{\Omega})} t^{\ip}\,u^{**}(t),
\end{equation}
respectively.
It follows from \eqref{E:f**-subadditivity} that
$\|\cdot\|_{L^{(p,\infty)}(\Omega)}$ is a norm for any $p>0$. On the other hand,
$\|\cdot\|_{L^{p,\infty}(\Omega)}$ is in general only a quasinorm. Indeed, by \eqref{E:f*-subadditivity}, we have for $u,v\colon\Omega\to\R$ measurable,
\begin{equation}\label{E:quasinorm-property-of-weak-lebesgue}
	\nrm{u+v}_{L^{p,\infty}(\Omega)}
		\le \sup_{t\in(0,\abs{\Omega}/2)} (2t)^\ip\bigl( u^*(t) + v^*(t)\bigr)
		\le 2^\ip\bigl( \nrm{u}_{L^{p,\infty}(\Omega)} + \nrm{v}_{L^{p,\infty}(\Omega)} \bigr).
\end{equation}
The functional $\|\cdot\|_{L^{p,\infty}(\Omega)}$ is a norm only if $p=\infty$
and it is equivalent to a norm if $1<p<\infty$ (see \eg~\cite[Theorem~8.2.2 and
Corollary~8.2.4]{Pic:13}). Important particular instances of Lorentz spaces are
the cornerstone spaces $L^1(\Omega)$ and $L^\infty(\Omega)$. Indeed, for
$p\in(0,1]$, one has by monotonicity
\begin{equation}\label{E:L1-is-lorentz}
	\nrm{u}_{L^{(p,\infty)}(\Omega)}
		= \sup_{t\in(0,\abs\Omega)}t^{\ip-1}\int_{0}^{t}u^{*}(s)\d s
		= \abs\Omega^{\ip-1} \nrm{u}_{L^{1}(\Omega)}
		\quad \text{for every measurable $u$ on $\Omega$,}
\end{equation}
hence $L^1(\Omega)=L^{1,1}(\Omega)=L^{(p,\infty)}(\Omega)$.
Moreover, obviously,
$L^\infty(\Omega)=L^{\infty,\infty}(\Omega)=L^{(\infty,\infty)}(\Omega)$.

Our first result reads as follows.

\begin{theorem} \label{T:weak-lebesgue-non-disjointly-superadditive}
Let $\Omega\subset\rn$, $n\in\N$, be a nonempty set of positive measure
and let $r\in(0,\infty]$. Then the space
$L^{r,\infty}(\Omega)$ is not disjointly superadditive.
\end{theorem}

\begin{proof}
Let $m\in\N$ be given. There are pairwise disjoint
measurable subsets $E_k$, $k=1,\dots,m$, of $\Omega$ such that their measures
$s_k=\abs{E_k}$ satisfy
\begin{equation} \label{E:sk}
	s_{k+1}\le \frac{s_k}{2}
		\quad\text{for each $k=1,\dots,m-1$}.
\end{equation}
We define the functions
$u_k = s_k^{-\ir} \chi_{E_k}$ for $k=1,\dots,m$.
Then
\begin{equation} \label{E:fk-rearrangement}
	u_k^* = s_k^{-\ir} \chi_{(0,s_k)},
		\quad\text{for $k=1,\dots,m$}
\end{equation}
and, consequently,
\begin{equation} \label{E:fk-norm}
	\nrm{u_k}_{L^{r,\infty}(\Omega)}
		= \sup_{t\in(0,\abs{\Omega})} t^{\ir}\, u_k^*(t)
		= 1.
\end{equation}
Let us denote
\begin{equation} \label{E:f-def}
	u = \sum_{k=1}^m u_k.
\end{equation}
Since the functions $u_k$ have pairwise disjoint supports, we have
\begin{equation} \label{E:fks-rearrangement}
	u^*
		= \sum_{k=1}^{m} s_k^{-\ir} \chi_{(a_k,a_{k-1})},
\end{equation}
where $a_k=s_{k+1}+\dots+s_m$ for $k=0,\dots,m-1$
and $a_m=0$
(see Figure~\ref{F:fks-rearrangement}).
\begin{figure}[b!]
\centering
\begin{tikzpicture}
\begin{axis}[
	clip=false,
	x=80mm, y=30mm,
	axis x line=middle,
	axis y line=middle,
	axis line style={-},
	xtick=\empty,
	extra x ticks={0,0.06,0.18,0.45,1},
	extra x tick labels={$a_m$,,,$a_1$,$a_0$},
	ytick=\empty,
	extra y ticks={0.2,1},
	extra y tick labels={$s_1^{-\ir}$,$s_m^{-\ir}$},
	ymin=-0.03, ymax=1.03,
	xmin=-0.03, xmax=1.03,
	]
	\addplot[ycomb, dashed]
		coordinates {(0.06,1) (0.18,0.75) (0.45,0.45) (1,0.2)};
	\addplot[jump mark left, ultra thick]
		coordinates { (0,1) (0.06,0.75) (0.18,0.45) (0.45,0.2) (1,0.2)};
	\node[below] at (axis cs: 0.725,0) {$\underbrace{\phantom{\hbox to 40mm{}}}_{s_1}$};
	\node[below] at (axis cs: 0.12,0) {$\underbrace{\phantom{\hbox to 1mm{}}}_{s_{m-1}}$};
	\node[above right] at (axis cs: 1,0.2) {$u^*$};
\end{axis}
\end{tikzpicture}
\caption{Rearrangement of the sum $\sum u_k$.}
\label{F:fks-rearrangement}
\end{figure}
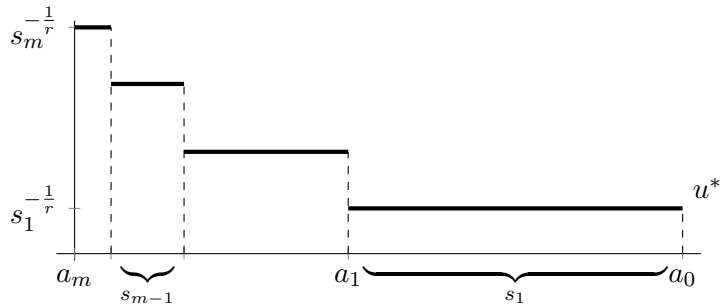
Consequently,
\begin{align} \label{E:f-norm}
	\begin{split}
	\nrm{u}_{L^{r,\infty}(\Omega)}
		& = \sup_{t\in(0,\abs\Omega)} t^\ir\, u^*(t)
			= \adjustlimits\max_{j\in\{1,\dots,m\}} \sup_{t\in(a_j,a_{j-1})}
			\sum_{k=1}^m t^\ir s_k^{-\ir} \chi_{(a_k,a_{k-1})}(t)
			\\
		& = \max_{j\in\{1,\dots,m\}} a_{j-1}^\ir s_j^{-\ir}
			\le \max_{j\in\{1,\dots,m\}} (2s_j)^\ir s_j^{-\ir}
			= 2^\ir,
	\end{split}
\end{align}
where we have used property \eqref{E:sk} to show
\begin{equation} \label{E:aj2sj}
	a_{j-1} = s_j+\cdots+s_m \le 2s_j
		\quad\text{for $j=1,\dots,m$}.
\end{equation}
Now suppose that the space $L^{r,\infty}(\Omega)$ has the disjoint
superadditivity property. Then there exists $\gamma>0$ and
$C>0$ such that
\begin{equation*}
	\sum_{k=1}^{m} \nrm{u_k}_{L^{r,\infty}(\Omega)}^\gamma
		\le C \nrm{u}_{L^{r,\infty}(\Omega)}^\gamma,
\end{equation*}
that is, by \eqref{E:fk-norm} and \eqref{E:f-norm}, $m \le C2^{\gamma/r}$,
which is clearly absurd because $m$ was selected
arbitrarily at the beginning.
\end{proof}

If we consider the functional \eqref{E:L(pq)-def} instead of \eqref{E:Lpq-def},
then we obtain a slightly different result.

\begin{theorem} \label{T:weak-lebesgue-non-disjointly-superadditive-2}
Let $\Omega\subset\rn$, $n\in\N$, be a nonempty set of positive measure.
Then the space $L^{(r,\infty)}(\Omega)$ is disjointly superadditive
if and only if $r\in(0,1]$.
\end{theorem}

\begin{proof}
Assume that $r\in(0,1]$. Then, by~\eqref{E:L1-is-lorentz},
$L^{(r,\infty)}(\Omega)=L^{1}(\Omega)$ with equivalent norms, hence
$L^{(r,\infty)}(\Omega)$ is disjointly superadditive.

Assume that $r\in(1,\infty]$.  Let $m\in\N$ be given and suppose that the sets
$E_k$, $k=1,\ldots,m$, are chosen as in the proof of
Theorem~\ref{T:weak-lebesgue-non-disjointly-superadditive}.  Let the functions
$u_k$ for $k=1,\dots,m$, the numbers $s_k$ for $k=1,\dots,m$, and the numbers
$a_j$ for $j=0,\dots,m$, have the same meaning as in the proof of
Theorem~\ref{T:weak-lebesgue-non-disjointly-superadditive}.
By~\eqref{E:fk-rearrangement},
\begin{equation*}
	u_k^{**}(t)
		= s_k^{-\ir} \chi_{(0,s_k)}(t)
			+ \frac{s_k^{1-\ir}}{t} \chi_{[s_k,\abs{\Omega})}(t)
		\quad\text{for $t\in(0,\abs{\Omega})$ and $k=1,\dots,m$},
\end{equation*}
and therefore
\begin{align*}
	\nrm{u_k}_{L^{(r,\infty)}(\Omega)}
		= \sup_{t\in(0,\abs{\Omega})} t^\ir\,u_k^{**}(t)
		= \max\left\{s_k^{-\ir} \sup_{t\in(0,s_k)} t^\ir,
				\;
				s_k^{1-\ir} \sup_{t\in[s_k,\abs{\Omega})} t^{\ir-1}\right\}
		= 1
\end{align*}
for every $k=1,\dots,m$.
Denote by $u$ the sum of all $u_k$'s as in \eqref{E:f-def},
fix $j\in\{1,\dots,m-1\}$ and $t\in(a_j,a_{j-1}]$. Then,
using \eqref{E:fks-rearrangement},
\begin{equation*}
	\int_0^t u^*(s)\,\d s
		= \int_{0}^{a_j} u^*(s)\,\d s
			+ \int_{a_j}^{t} u^*(s)\,\d s
		= \sum_{k=j+1}^m s_k^{1-\ir} + (t-a_j) s_j^{-\ir},
\end{equation*}
and therefore
\begin{align} \label{E:sup-ajs}
	\begin{split}
	\sup_{t\in(a_j,a_{j-1}]} t^\ir\,u^{**}(t)
		& \le \sup_{t\in(a_j,a_{j-1}]}
			t^{\ir-1} \left( \sum_{k=j+1}^{m} s_k^{1-\ir} + t s_j^{-\ir}\right)
		\le a_j^{\ir-1} \sum_{k=j+1}^{m} s_k^{1-\ir} + a_{j-1}^\ir s_j^{-\ir}.
	\end{split}
\end{align}
Due to \eqref{E:aj2sj}, we have $a_{j-1}^\ir s_j^{-\ir}\le 2^\ir$.
By the definition of $a_j$, it is $a_j\ge s_{j+1}$ and, using \eqref{E:sk}, we have
\begin{equation*}
	a_j^{\ir-1} \sum_{k=j+1}^{m} s_k^{1-\ir}
		\le a_j^{\ir-1} s_{j+1}^{1-\ir} \sum_{k=j+1}^{m} \left(2^{1-\ir}\right)^{j+1-k}
		\le 2^{1-\ir}.
\end{equation*}
Altogether, \eqref{E:sup-ajs} yields
\begin{equation} \label{E:sup1}
	\sup_{t\in(a_j,a_{j-1}]} t^\ir\, u^{**}(t)
		\le 2^\ir + 2^{1-\ir}
		\le 4
	\quad\text{for each $j=1,\dots, m-1$}.
\end{equation}
It remains to consider the case when $t\in(0,a_{m-1}]$. But, for such $t$,
\begin{equation*}
	\int_0^t u^*(s)\,\d s
		= t s_m^{-\ir},
\end{equation*}
and therefore
\begin{equation} \label{E:sup2}
	\sup_{t\in(0,a_{m-1}]} t^\ir\,u^{**}(t)
		= \sup_{t\in(0,a_{m-1}]} t^\ir s_m^{-\ir}
		= a_{m-1}^\ir s_m^{-\ir}
		= 1.
\end{equation}
Estimates \eqref{E:sup1} and \eqref{E:sup2} combined give
$\nrm{u}_{L^{(r,\infty)}(\Omega)}\le 4$.
Now, assuming that $L^{(r,\infty)}(\Omega)$ has the disjoint
superadditivity property for some $\gamma>0$ and $C>0$, we conclude
that $m\le C4^\gamma$, which is impossible, as $m$ was arbitrary.
\end{proof}

\begin{remark}
The only property of the measure space we required in the proof was that there
exist pairwise disjoint sets of prescribed measure.
Theorems~\ref{T:weak-lebesgue-non-disjointly-superadditive}
and \ref{T:weak-lebesgue-non-disjointly-superadditive-2} thus remain valid
even if we replace the set $\Omega\subset\rn$ by any non-atomic $\sigma$-finite
measure space.
\end{remark}

\section{Shrinking property}

A key step to the maximal non-compactness is shrinking property of an
embedding.

\begin{definition}
Let $X(\Omega)$ and $Y(\Omega)$ be quasinormed linear spaces of functions defined on
$\Omega\subset\rn$.  We say that the embedding $I\colon X(\Omega)\to Y(\Omega)$
has \emph{shrinking property} if for every nonempty open set $G\subset\Omega$ one
has
\begin{equation*}
	\nrm{I} = \sup\{\nrm{u}_{Y(\Omega)}: \nrm{u}_{X(\Omega)}\le 1,\ \sprt u\subset G\}.
\end{equation*}
\end{definition}

In other words, shrinking property guarantees that the norm of an embedding
is attained by functions having their support in an arbitrary nonempty open set $G$.

\begin{proposition}\label{P:shrinking-sobolev}
Let $\Omega\subset\rn$, $n\in\N$, be open bounded and
nonempty set and let $k\in\N$, $k<n$.
Assume that $X(\Omega)$ is either
$L^{p,q}(\Omega)$ or $L^{(p,q)}(\Omega)$ with $p\in[1,\nk)$ and
$q\in[1,\infty]$ and let $Y(\Omega)$ be one of the spaces $L^{p^*,\infty}(\Omega)$ or
$L^{(p^*,\infty)}(\Omega)$, where $p^*=\frac{np}{n-kp}$.  Then the embedding
\begin{equation} \label{E:emb}
	I\colon V_0^k X(\Omega)\to Y(\Omega)
\end{equation}
has shrinking property.
\end{proposition}

\begin{proof}
Let $G\subset\Omega$ be open and nonempty.
Since $\Omega$ is bounded, open and nonempty, there are two concentric
balls $B_1$ and $B_2$ obeying $B_1\subset G\subset\Omega\subset B_2$.
It holds
\begin{equation} \label{E:norms-are-the-same}
	\sup_{u\ne 0} \frac{\nrm{u}_{Y(B_1)}}{\nrm{u}_{V_0^kX(B_1)}}
		\le \sup_{\substack{u\ne 0\\\sprt u\subset G}}
			\frac{\nrm{u}_{Y(\Omega)}}{\nrm{u}_{V_0^kX(\Omega)}}
		\le \sup_{u\ne 0} \frac{\nrm{u}_{Y(\Omega)}}{\nrm{u}_{V_0^kX(\Omega)}}
		\le \sup_{u\ne 0} \frac{\nrm{u}_{Y(B_2)}}{\nrm{u}_{V_0^kX(B_2)}},
\end{equation}
since every $u\in V_0^kX(B_1)$ extended by zero on $\Omega\setminus B_1$
belongs to $V_0^kX(\Omega)$
and every $u\in V_0^kX(\Omega)$ extended
by zero on $B_2\setminus\Omega$ is an element of $V_0^kX(B_2)$.
We show that in fact the equalities in \eqref{E:norms-are-the-same} hold.
By the translation invariance of the Lorentz functionals, we may
assume that both the balls are centered at the origin.
Let us denote $r_1$ and $r_2$ the radii of $B_1$ and $B_2$, respectively.
For $u\in V_0^kX(B_2)$, set $\kappa=r_2/r_1$ and define
\begin{equation*}
	u_{\kappa}(x) = u(\kappa x)
		\quad \text{for $x\in B_1$.}
\end{equation*}
Clearly $u_{\kappa}$ is a Sobolev function on $B_1$.
We will show that
\begin{equation} \label{E:ratios-are-the-same}
	\frac{\nrm{u_{\kappa}}_{Y(B_1)}}{\nrm{u_{\kappa}}_{V_0^kX(B_1)}}
		= \frac{\nrm{u}_{Y(B_2)}}{\nrm{u}_{V_0^kX(B_2)}}.
\end{equation}
First, observe that
\begin{equation} \label{E:star}
	(u_{\kappa})^*(t) = u^*(\kappa^nt)
\end{equation}
and
\begin{equation} \label{E:doublestar}
	(u_{\kappa})^{**}(t) = u^{**}(\kappa^nt)
\end{equation}
for $t>0$. Consequently,
\begin{equation} \label{E:D-star}
	(D^\multiindex u_{\kappa})^*(t) = \kappa^{\abs{\multiindex}} (D^\multiindex u)^*(\kappa^nt)
\end{equation}
and
\begin{equation} \label{E:D-doublestar}
	(D^\multiindex u_{\kappa})^{**}(t) = \kappa^{\abs{\multiindex}} (D^\multiindex u)^{**}(\kappa^nt)
\end{equation}
for $t\in(0,\abs{B_1})$ and for each multiindex $\multiindex$.
Now, assume that $q=\infty$ and $X=L^{p,\infty}$. We have
\begin{align*}
	\nrm{u_{\kappa}}_{V^{k}_0L^{p,\infty}(B_1)}
		& = \sum_{\abs{\multiindex} = k} \nrm{D^{\multiindex} u_{\kappa}}_{L^{p,\infty}(B_1)}
			= \sum_{\abs{\multiindex} = k} \sup_{t\in(0,\abs{B_1})}
				t^{\frac1p} (D^{\multiindex}u_{\kappa})^{*}(t)
			\\
		& = \sum_{\abs{\multiindex} = k} \kappa^{\abs{\multiindex}}
				\sup_{t\in(0,\abs{B_1})} t^{\frac1p} (D^{\multiindex}u)^{*}(\kappa^nt)
			= \kappa^{k-\frac np} \sum_{\abs{\multiindex} = k}
				\sup_{s\in(0,\abs{B_2})} s^{\frac1p} (D^{\multiindex}u)^{*}(s)
			\\
		& = \kappa^{k-\frac np} \nrm{u}_{V^k_0L^{p,\infty}(B_2)},
\end{align*}
where we used \eqref{E:D-star} and the change of variables $s=\kappa^nt$.
If $X=L^{(p,\infty)}$, we use \eqref{E:D-doublestar} instead. For $q<\infty$,
analogous computations are shown in \citep[Theorem~1.2]{Bou:19}. Thus,
\begin{equation} \label{E:norms-u-hatu}
	\nrm{u_{\kappa}}_{V^{k}_0X(B_1)}
		= \kappa^{k-\frac np} \nrm{u}_{V^k_0X(B_2)}.
\end{equation}
Similarly, if $Y=L^{p^*,\infty}$, we have
\begin{align*}
	\nrm{u_{\kappa}}_{L^{p^*,\infty}(B_1)}
		& = \sup_{t\in(0,\abs{B_1})} t^{\frac{1}{p^*}} (u_{\kappa})^*(t)
			= \sup_{t\in(0,\abs{B_1})} t^{\frac{1}{p^*}} u^*(\kappa^nt)
			\\
		& = \kappa^{-\frac{n}{p^*}} \sup_{s\in(0,\abs{B_2})} s^{\frac{1}{p^*}} u^*(s)
			= \kappa^{k-\frac{n}{p}} \nrm{u}_{L^{p^*,\infty}(B_2)},
\end{align*}
where we substituted $s=\kappa^nt$ and used the relation
\begin{equation*}
	-\frac{n}{p^*} = k-\frac np.
\end{equation*}
For $Y=L^{(p^*,\infty)}$, we use \eqref{E:doublestar} instead of \eqref{E:star}.
Altogether,
\begin{equation} \label{E:norms-u-hatu2}
	\nrm{u_{\kappa}}_{Y(B_1)}
		= \kappa^{k-\frac np} \nrm{u}_{Y(B_2)}
\end{equation}
and \eqref{E:ratios-are-the-same} follows by \eqref{E:norms-u-hatu} and
\eqref{E:norms-u-hatu2}.
\end{proof}

\section{Maximal non-compactness of a Sobolev embedding into weak Lebesgue space}
\label{S:maximal-non-compactness}

In this section, we focus on Question~\ref{Q:1}.

\begin{theorem} \label{T:max-non-compact-general}
Let $\Omega\subset\rn$, $n\in\N$, be open nonempty set of positive measure and
let $r\in(0,\infty)$.  Assume that $X(\Omega)$ is a quasinormed linear space
containing measurable functions and $Y(\Omega)$ is either
$L^{r,\infty}(\Omega)$ or $L^{(r,\infty)}(\Omega)$. If $I\colon X(\Omega)\to
Y(\Omega)$ has shrinking property, then $I$ is maximally non-compact.
\end{theorem}

\begin{proof}
Assume the contrary, that is $\bmc(I)<\nrm{I}$,
and set $\rpar$ and $\varepsilon>0$ such that $\bmc(I)<\rpar<\rpar+2\varepsilon<\nrm{I}$.
There exist $j\in\N$ and functions
$v_1,\dots,v_j$ in $Y(\Omega)$ such that
\begin{equation}\label{E:union}
	B_{X(\Omega)}
		\subset \bigcup_{k=1}^{j}
			\left(v_k+\rpar B_{Y(\Omega)}\right).
\end{equation}
We may assume that
\begin{equation}\label{E:out}
	\nrm{v_k}_{Y(\Omega)}\le 2^{1+\ir} \nrm{I}
		\quad\text{for each $k=1,\dots,j$}.
\end{equation}
Indeed, suppose that $\nrm{v_k}_{Y(\Omega)}>2^{1+\ir}\nrm{I}$ for some $k$.
Recall that $\|\cdot\|_{L^{(r,\infty)}}$ is a norm
while, by~\eqref{E:quasinorm-property-of-weak-lebesgue}, $\|\cdot\|_{L^{r,\infty}}$ is a quasinorm with the constant $2^{1/r}$. In any case, every $v\in B_{X(\Omega)}$
obeys $\nrm{v}_{Y(\Omega)}\le \nrm{I}$ and
\begin{equation*}
	\nrm{v_k-v}_{Y(\Omega)}
		\ge 2^{-\ir}\nrm{v_k}_{Y(\Omega)} - \nrm{v}_{Y(\Omega)}
		> 2\nrm{I} - \nrm{I}
		>\rpar.
\end{equation*}
In other words,
\begin{equation*}
	B_{X(\Omega)}
		\cap \left(v_k+\rpar B_{Y(\Omega)}\right)
		= \emptyset,
\end{equation*}
and so the function $v_k$ can be excluded from the collection on the right-hand
side of~\eqref{E:union}.

If $Y=L^{r,\infty}$, we choose $\eta\in(0,1)$ such that
\begin{equation} \label{E:eta-def}
	0 \ge \rpar\left( 1-\eta^{-\ir} \right)
		\ge - \varepsilon
\end{equation}
and set $\eta=0$ if $Y=L^{(r,\infty)}$. Fix $\ell\in\N$ such that
\begin{equation} \label{E:ell-def}
	\bigl(\ell(1-\eta)\bigr)^\ir
		\ge \frac{2^{1+\ir}}{\varepsilon} \nrm{I}
\end{equation}
and set $m = j \ell$ and $\omega=\abs{\Omega}/(\ell(1-\eta))$.
Denote by $B_1,\dots,B_m$ pairwise disjoint balls of the same radii, centered
at $x_1,\dots,x_m$, respectively, and all contained in $\Omega$.
We may without loss of generality assume that $\abs{B_k}<\omega$
for all $k=1,\dots,m$.
By the shrinking property, there exists $u_1\in X(\Omega)$ such that $\nrm{u_1}_{X(\Omega)}=1$, $\sprt u_1\subset B_1$
and $\nrm{u_1}_{Y(\Omega)}>\rpar+2\varepsilon$.
For each $k=2,\dots,m$, we define $u_k\colon\Omega\to\R$ by
\begin{equation*}
	u_k(x) = u_1(x + x_1 - x_k)\chi_{B_k}(x)
		\quad\text{for $x\in\Omega$}.
\end{equation*}
Then all the functions $u_k$ are equimeasurable, \ie
\begin{equation}\label{E:same-rear}
	u_k^* = u_i^* \quad\text{for every $k,i=1,\dots,m$},
\end{equation}
one has $u_k\in X(\Omega)$, $\nrm{u_k}_{X(\Omega)}=1$, $\sprt u_k\subset B_k$, and
\begin{equation} \label{E:fk-lower}
	\nrm{u_k}_{Y(\Omega)}
		> \rpar+2\varepsilon
		\quad\text{for each $k=1,\dots,m$}.
\end{equation}
Due to \eqref{E:union}, it holds that
\begin{equation} \label{E:fk-subset}
	\{u_1,\dots, u_m\}
		\subset B_{X(\Omega)}
		\subset \bigcup_{k=1}^{j}
			\left(v_k+\rpar B_{Y(\Omega)}\right).
\end{equation}
By the Pigeonhole principle, at least one of the balls in the union on the rightmost
side of \eqref{E:fk-subset}
must contain at least $m/j=\ell$ functions from $\{u_1,\dots, u_m\}$.
More precisely, there exist $i\in\{1,\dots,j\}$
and distinct functions
$\{u^1,\dots,u^{\ell}\}\subset\{u_1,\dots,u_m\}$
such that $u^k\in v_i + \rpar B_{Y(\Omega)}$
for every $k=1,\dots,\ell$. That is,
\begin{equation} \label{E:g-sigma-f-k}
	\nrm{v_{i}-u^k}_{Y(\Omega)}
		< \rpar
			\quad\text{for every $k=1,\dots,\ell$}.
\end{equation}
Let us denote
\begin{equation*}
	v^k = v_i\, \chi_{\sprt u^k}
		\quad\text{for $k=1,\dots,\ell$}
\end{equation*}
and $v = \sum_{k=1}^\ell v^k$.
Next, define
\begin{equation*}
	w^k(x)=
	\begin{cases}
		v^k(x) & \text{if $\abs{u^k(x)} \ge \abs{v^k(x)}$}
			\\
		u^k(x) & \text{if $\abs{u^k(x)} < \abs{v^k(x)}$}
	\end{cases}
		\quad\text{for $x\in\Omega$ and $k=1,\dots,\ell$}
\end{equation*}
and also $w=\sum_{k=1}^\ell w^k$.
Then $\abs{w}\le\abs{v}$, whence, thanks to the well-known lattice property of the weak Lebesgue functionals
\eqref{E:weak-lebesgue-norms}, by the definition of $v$, and by~\eqref{E:out}, we get
\begin{equation} \label{E:h-upper-bound}
	\nrm{w}_{Y(\Omega)}
		\le \nrm{v}_{Y(\Omega)}
		\le \nrm{v_{i}}_{Y(\Omega)}
		\le 2^{1+\ir}\nrm{I}.
\end{equation}
Since $w^{k}-u^{k}$ equals either zero or $v^{k}-u^{k}$, one clearly has
$\abs{w^{k}-u^{k}} \le \abs{v^{k}-u^{k}}$. Furthermore, by the definition of $v^{k}$,
$v^{k}-u^{k}$ equals either zero or $v_{i}-u^{k}$, hence $\abs{v^{k}-u^{k}} \le
\abs{v_{i}-u^{k}}$. Altogether, calling into play again the lattice property of
$\|\cdot\|_{Y(\Omega)}$ and finally using \eqref{E:g-sigma-f-k}, we obtain
\begin{equation} \label{E:hk-gk}
	\nrm{w^k-u^k}_{Y(\Omega)}
		\le \nrm{v^k-u^k}_{Y(\Omega)}
		\le \nrm{v_{i}-u^k}_{Y(\Omega)}
		< \rpar
\end{equation}
for every $k=1,\dots,\ell$.

Now, assume that $Y=L^{(r,\infty)}$.
By \eqref{E:fk-lower}, there is some $t_0\in(0,\omega)$ such that
\begin{equation*}
	(u^k)^{**}(t_0)
		> (\rpar+2\varepsilon)\,t_0^{-\ir}
		\quad\text{for every $k=1,\dots,\ell$}.
\end{equation*}
It is important to notice that, thanks to~\eqref{E:same-rear}, $t_0$ is
independent of $k$.
Since the measure is non-atomic, for each
$k=1,\dots,\ell$ there exists $E^k\subset\sprt u^k$ such
that $\abs{E^k}=t_0$ and
\begin{equation*}
	\frac{1}{t_0}\int_{E^k} \abs{u^k}
		> \left(\rpar+2\varepsilon\right) t_0^{-\ir}
\end{equation*}
due to \eqref{E:maximal-alt}.  By the Hardy--Littlewood inequality
\cite[Chapter~2, Theorem~2.2]{BS:88} and~\eqref{E:hk-gk},
\begin{equation*}
	\frac{1}{t_0}\int_{E^k} \abs{u^k-w^k}
		\le (u^k-w^k)^{**}(t_0)
		\le \rpar t_0^{-\ir}
			\quad\text{for every $k=1,\dots,\ell$}.
\end{equation*}
Altogether, we have
\begin{align*}
	\frac{1}{t_0}\int_{E^k} \abs{w^k}
		& = \frac{1}{t_0}\int_{E^k} \abs{u^k-(u^k-w^k)}
		\ge \frac{1}{t_0}\int_{E^k} \abs{u^k}
				- \frac{1}{t_0}\int_{E^k} \abs{u^k-w^k}
			\\
		& > \left(\rpar + 2\varepsilon\right) t_0^{-\ir}
			- \rpar t_0^{-\ir}
		= 2\varepsilon t_0^{-\ir}
			\quad\text{for every $k=1,\dots,\ell$}.
\end{align*}
Denote $E=\bigcup_{k=1}^{\ell} E^k$. Since the sets
$E^k$ are pairwise disjoint, we get $\abs{E}=\ell t_0$ and
\begin{equation*}
	w^{**}(\ell t_0)
		\ge \frac{1}{\ell t_0}\int_{E} \abs{w}
		= \frac{1}{\ell t_0} \sum_{k=1}^{\ell}\int_{E^k} \abs{w^k}
		> 2\varepsilon t_0^{-\ir},
\end{equation*}
which gives
\begin{equation} \label{E:h-lower-bound}
	\nrm{w}_{L^{(r,\infty)}(\Omega)}
		\ge (\ell t_0)^{\ir}\, w^{**}(\ell t_0)
		> 2\varepsilon \ell^{\ir}
		> 2^{1+\ir} \nrm{I},
\end{equation}
thanks to~\eqref{E:ell-def}. Recall that $\eta=0$ in this case.
Estimate \eqref{E:h-lower-bound}
now contradicts~\eqref{E:h-upper-bound}.

Let $Y=L^{r,\infty}$ instead.
By \eqref{E:fk-lower},
there is $t_0\in(0,\omega)$ such that
\begin{equation} \label{E:fk-lower-a}
	(u^k)^{*}(t_0)
		> (\rpar+2\varepsilon)\, t_0^{-\ir}
		\quad\text{for every $k=1,\dots,\ell$}.
\end{equation}
Using \eqref{E:f*-subadditivity}, we have
for all $k=1,\dots,\ell$
\begin{align*}
	(w^k)^*(t_0-\eta t_0)
		& \ge (u^k)^*(t_0) - (u^k-w^k)^*(\eta t_0)
			> (\rpar+2\varepsilon)\, t_0^{-\ir}
				- \rpar (\eta t_0)^{-\ir}
			\\
		& = \rpar\left( 1-\eta^{-\ir} \right) t_0^{-\ir}
				+ 2\varepsilon t_0^{-\ir}
			\ge \varepsilon t_0^{-\ir},
\end{align*}
where we have used \eqref{E:hk-gk}, \eqref{E:fk-lower-a} and \eqref{E:eta-def}. Now, since $w^k$'s have pairwise disjoint supports, one has
\begin{equation*}
	w^*\bigl(\ell(1-\eta)t_0\bigr)
		\ge \varepsilon t_0^{-\ir},
\end{equation*}
whence
\begin{equation} \label{E:h-lower-bound-2}
	\nrm{w}_{L^{r,\infty}(\Omega)}
		\ge \bigl(\ell(1-\eta)t_0\bigr)^\ir w^{*}\bigr(\ell(1-\eta)t_0\bigr)
		> \varepsilon \bigl(\ell(1-\eta)\bigr)^\ir
		> 2^{1+\ir} \nrm{I},
\end{equation}
where the last inequality
is due to \eqref{E:ell-def}. Finally, \eqref{E:h-lower-bound-2}
contradicts~\eqref{E:h-upper-bound}.
\end{proof}

Theorem~\ref{T:max-non-compact-general} combined with
Proposition~\ref{P:shrinking-sobolev} leads to the following result, whose
special case immediately answers Question~\ref{Q:1}.

\begin{theorem} \label{T:max-non-compact}
Let $\Omega\subset\rn$, $n\in\N$, be open bounded and nonempty set and let
$k\in\N$, $k<n$.  Suppose that $X(\Omega)$ is either $L^{p,q}(\Omega)$ or
$L^{(p,q)}(\Omega)$ in which $p\in[1,\nk)$, $q\in[1,\infty]$, and $Y(\Omega)$
is either $L^{p^*,\infty}(\Omega)$ or $L^{(p^*,\infty)}(\Omega)$, where
$p^*=\frac{np}{n-kp}$.  Then the Sobolev embedding \eqref{E:emb} is maximally
non-compact.
\end{theorem}

\section{Embeddings into the space of essentially bounded functions}
\label{S:embeddings-into-l-infty}

In this section, we shall exhibit that unlimited supply of disjointly supported
functions of the same norm (or, in particular, the shrinking property) on its
own is not enough to guarantee maximal non-compactness of an embedding. To this
end, we shall investigate embeddings into $L^{\infty}(\Omega)$ here. We begin
by introducing a new quantity assigned to such an embedding, which will prove of
substantial use later.  Let $X(\Omega)$ be a quasinormed linear space of
measurable functions defined on $\Omega$ and consider the identity operator
\begin{equation} \label{E:emb-to-ell-infty}
	I\colon X(\Omega)\to L^\infty(\Omega).
\end{equation}
We define the \emph{span} of $I$ by
\begin{equation*}
	\sigma(I) = \sup\left\{\esssup u-\essinf u:
		u\in X(\Omega), \ \nrm{u}_{X(\Omega)}\le 1\right\}.
\end{equation*}
Trivial inspection shows that
\begin{equation} \label{E:span-estimates-trivial}
	\sigma(I) \le 2\nrm{I}.
\end{equation}
We shall show that if inequality~\eqref{E:span-estimates-trivial} is strict and
the domain space has the unlimited supply (shrinking) property, then $I$ is not
maximally non-compact.  The principal idea in the background of this result is
rather neatly illustrated with the example which was mentioned in the
introductory section.  A key step to the result is the following general
assertion.  It requires a simple assumption on
embedding~\eqref{E:emb-to-ell-infty}, which reads
\begin{equation} \label{E:span-estimate-nontrivial}
	\nrm{I} \le \sigma(I).
\end{equation}
This hypothesis prevents the space $X(\Omega)$ from being too ``poor'', like
constant functions on $\Omega$, for instance.  Observe that shrinking property
is a sufficient condition for $X(\Omega)$ to obey
\eqref{E:span-estimate-nontrivial}.

\begin{proposition}\label{P:beta-upper-of-X-to-ell-infty}
Let $\Omega\subset\rn$, $n\in\N$, be a nonempty set of positive measure.
Assume that the embedding $I$ in \eqref{E:emb-to-ell-infty} obeys
\eqref{E:span-estimate-nontrivial}.  The measure of non-compactness of $I$
satisfies
\begin{equation}\label{E:I-rho}
    \bmc(I)\le \frac{\sigma(I)}{2}.
\end{equation}
\end{proposition}

\begin{proof}
Suppose that $\rpar>\sigma(I)/2$ and let $m\in\N$ be such that
\begin{equation*}
    \frac{\sigma(I)}{m}<\rpar-\frac{\sigma(I)}{2}.
\end{equation*}
Set
\begin{equation*}
	\lambda_k=\frac{k\sigma(I)}{2m}
		\quad \text{for $k=-m,\dots,m$}.
\end{equation*}
Observe that
\begin{equation}\label{E:difference-between-lambdas}
	\lambda_k-\lambda_{k-1}<\rpar-\frac{\sigma(I)}{2}
		\quad \text{for $k=-m+1,\dots,m$}.
\end{equation}
Define
\begin{equation*}
	v_k(x)=\lambda_k
		\quad \text{for $x\in\Omega$ and $k=-m,\dots,m$.}
\end{equation*}
Then of course each $v_k$ belongs to $L^{\infty}(\Omega)$. Now let $u\in
B_{X(\Omega)}$. Then, by \eqref{E:span-estimate-nontrivial}, we have
\begin{equation} \label{E:u-upper-bound}
    \nrm{u}_{L^{\infty}(\Omega)}\le\nrm{I} \le \sigma(I)
\end{equation}
and, by the definition of $\sigma(I)$,
\begin{equation} \label{E:u-span}
	\esssup u - \essinf u \le \sigma(I).
\end{equation}
If $\essinf u=-\sigma(I)$, then, by \eqref{E:u-span}, $u$ essentially ranges
between $-\sigma(I)$ and $0$, whence $u\in v_{-m}+\rpar B_{L^{\infty}(\Omega)}$.
Suppose that $\essinf u\in(-\sigma(I),0]$. We find $k\in\{-m+1,\dots,m\}$ such that
\begin{equation}\label{E:estimate-of-essinf-u}
	 \essinf u + \frac{\sigma(I)}{2}\in(\lambda_{k-1},\lambda_k].
\end{equation}
Then, by~\eqref{E:estimate-of-essinf-u} and \eqref{E:u-span} again,
\begin{equation*}
	\lambda_{k} + \rpar
		> \lambda_{k} + \frac{\sigma(I)}{2}
		\ge \essinf u + \sigma(I)
		\ge \esssup u.
\end{equation*}
On the other hand, by~\eqref{E:difference-between-lambdas},
\begin{equation*}
	\essinf u
		\ge \lambda_{k-1} - \frac{\sigma(I)}{2}
		>\lambda_{k} - \rpar.
\end{equation*}
Therefore,
\begin{equation*}
	\lambda_{k} - \rpar
		< \essinf u
		\le \esssup u
		< \lambda_{k} + \rpar,
\end{equation*}
which means that $u\in v_{k} + \rpar B_{L^{\infty}(\Omega)}$.  Finally, if
$\essinf u \in(0,\sigma(I)]$, then, due to \eqref{E:u-upper-bound}, $u$
essentially takes values between $0$ and $\sigma(I)$, hence $u\in
v_m+\rpar L^\infty(\Omega)$.  This shows that $\bmc(I)< \rpar$. Since
$\rpar>\sigma(I)/2$ was arbitrary, we get~\eqref{E:I-rho}.
\end{proof}

We shall now show that in the case when $X(\Omega)$ has the unlimited supply property,
the converse inequality to~\eqref{E:I-rho} holds as well. The principal task in
such a case is, given $\rpar>\bmc(I)$, to construct a large enough family of
disjointly supported functions such that the span of the difference of each two
of them exceeds $2\rpar$.

\begin{theorem} \label{T:beta-lower-of-X-to-ell-infty}
Let $\Omega\subset\rn$, $n\in\N$, be nonempty set of positive measure.  Suppose
that $\rpar>0$ and assume that, for any $\ell\in\N$, there exist pairwise
disjointly supported functions $u_k\in X(\Omega)$,
$k=1,\dots,\ell$, satisfying
\begin{equation} \label{E:X-norm-difference}
	\nrm{u_j-u_k}_{X(\Omega)} \le 1
		\quad\text{for distinct $j,k=1,\dots,\ell$}
\end{equation}
and
\begin{equation} \label{E:u-infty-norm-large}
	2\rpar > \esssup_{\Omega} u_k > \rpar
		\quad\text{for $k=1,\dots,\ell$.}
\end{equation}
Then, the measure of non-compactness of embedding $I$ in
\eqref{E:emb-to-ell-infty} obeys
\begin{equation} \label{E:betaI-lower-bound}
	\bmc(I)\ge\rpar.
\end{equation}
\end{theorem}

\begin{proof}
Assume that~\eqref{E:betaI-lower-bound} is not satisfied, \ie that
$\bmc(I)<\rpar$.  This means that there exist $m\in\N$ and a collection
$\{v_1,\dots,v_{m-1}\}\subset L^{\infty}(\Omega)$ such that
\begin{equation} \label{E:m-1-covering}
    B_{X(\Omega)}
			\subset\bigcup_{k=1}^{m-1}\left(v_k+\rpar B_{L^{\infty}(\Omega)}\right).
\end{equation}
Set $\ell=2^m$ and let $u_1,\dots,u_\ell$ be the sequence guaranteed by the
assumption of the theorem. Define
\begin{equation*}
    w_{i,j} = u_i - u_j
		\quad\text{for $1\le i<j\le \ell$}.
\end{equation*}
It follows from \eqref{E:X-norm-difference} that $w_{i,j}\in B_{X(\Omega)}$
and, due to \eqref{E:u-infty-norm-large},
\begin{equation}\label{E:distinction}
	\nrm{w_{i,j}-w_{i',j'}}_{L^{\infty}(\Omega)} > 2\rpar
	\quad\text{if and only if $i=j'$ or $j=i'$}
\end{equation}
for admissible indices.  Now, let $W_1,\dots,W_{m-1}$ be arbitrary pairwise
disjoint partitioning of the set
$\{w_{i,j}: 1\le i<j\le \ell\}$ satisfying
\begin{equation*}
	w_{i,j}\in W_k
		\quad\text{if}\quad
 	w_{i,j}\in v_k+\rpar B_{L^\infty(\Omega)}
	\quad\text{for $k=1,\dots,m-1$,}
\end{equation*}
which is possible due to \eqref{E:m-1-covering} and the fact that each
$w_{i,j}$ belongs to $B_X(\Omega)$.  Observe that if two functions $w$ and
$\tilde w$ share the same class, then there exists $k\in\{1,\dots,m-1\}$ such
that both $w$ and $\tilde w$ belong to the ball $v_k+\rpar B_{L^\infty(\Omega)}$,
whence
\begin{equation} \label{E:share-the-same-class}
    \nrm{w-\tilde w}_{L^{\infty}(\Omega)} \le 2\rpar.
\end{equation}
Our goal now is to show that such a partitioning is impossible with less than
$m$ classes, which would lead to a contradiction.

To have a better understanding of this setup, imagine that every $w_{i,j}$ is
represented by a field $(i,j)$ in a grid $2^m\times 2^m$. Thanks to the
constraint $1\le i<j\le 2^m$, we deal just with the lower triangle under the
diagonal (see Figure~\ref{fig:coloring:sufficiency}). A membership of $w_{i,j}$
to the class $W_c$ may be represented as a coloring of the corresponding field
$(i,j)$ by the color $c$.  The partitioning condition together with
\eqref{E:distinction} and~\eqref{E:share-the-same-class} translates to a simple
constraint: $i$-th line cannot share any color with $i$-th column for any
$i=1,2,\dots,2^m-1$.  The task is in showing that at least $m$ colors are
needed.

To this end, for each row $i\in\{1,\dots,2^m-1\}$, consider the sets
\begin{equation*}
    C_i=\{c: w_{i,j}\in W_c \text{ for some } i<j\le 2^m\},
\end{equation*}
that is, the colors contained in the $i$-th row.  We show that these sets need
to be distinct across rows.  Let $1\le i<j\le 2^m$ be given and let $c$ be the
class index such that  $w_{i,j}\in W_c$. Obviously $c\in C_i$. We claim that
$c\notin C_j$. Indeed, assume that there is some $j'$ such that $w_{j,j'}\in
W_c$. Then both $w_{i,j}$ and $w_{j,j'}$ share the same class (see
Figure~\ref{fig:coloring:necessity}) which means that
$\nrm{w_{i,j}-w_{j,j'}}_{L^{\infty}(\Omega)} \le 2\rpar$, due to
\eqref{E:share-the-same-class}.  This however
contradicts~\eqref{E:distinction} and therefore $C_i\neq C_j$. Since $i$ and $j$
were chosen arbitrarily, this shows that the family $\{C_1,\dots,C_{2^m-1}\}$
constitutes a collection of pairwise distinct nonempty subsets of the set of
all classes $\{1,\dots,m-1\}$, which is not possible.
\end{proof}

In order to apply Theorem~\ref{T:beta-lower-of-X-to-ell-infty} to the Sobolev
embedding \eqref{E:emb-limiting}, we need to show that the domain Sobolev space
has the unlimited supply of functions satisfying~\eqref{E:X-norm-difference}
and~\eqref{E:u-infty-norm-large}. To this end, we shall employ the shrinking
property.

\begin{proposition} \label{P:shrinking-emb-limiting}
Let $n\in\N$, $k\in\N$, $k\le n$, and let $\Omega\subset\rn$ be open bounded
and nonempty set.
Then the embedding from~\eqref{E:emb-limiting} has shrinking property.
\end{proposition}

\begin{proof}
Let $G\subset\Omega$ be nonempty open set. Suppose that $B_1$ and $B_2$ are concentric balls
such that $B_1\subset G\subset\Omega\subset B_2$. By the translation invariance,
we may assume that both the balls are centered at the origin.
Let $r_1$ and $r_2$ denote their radii
and set $\kappa=r_2/r_1$.
Given $u\in V^{k}_0L^{n/k,1}(B_2)$, define $u_\kappa(x)=u(\kappa x)$ for $x\in B_1$.
Then we have
\begin{equation*}
    \nrm{u_\kappa}_{L^{\infty}(B_1)}=\nrm{u}_{L^{\infty}(B_2)}
\end{equation*}
and, by a simple computation,
\begin{equation*}
    \nrm{u_\kappa}_{V^{k}_0L^{\nk,1}(B_1)}=\nrm{u}_{V^{k}_0L^{\nk,1}(B_2)}.
\end{equation*}
The assertion then follows by the same argument as in
Proposition~\ref{P:shrinking-sobolev}.
\end{proof}

Our next step will be exact evaluation of the span of the embedding operator
$I$ from~\eqref{E:emb-limiting}.

\begin{proposition}\label{P:span-of-emb-limiting}
Let $n,k\in\N$, $k\le n$, let $\Omega\subset\rn$ be open bounded and nonempty
set and let $I$ denote the embedding
from~\eqref{E:emb-limiting}. Then
\begin{equation}\label{E:comparison-of-I-and-L}
    \sigma(I) = 2^{1-\kn} \nrm I.
\end{equation}
\end{proposition}

\newcommand{\dom}[1][\Omega]{V_0^kL^{\nk,1}(#1)}
\begin{proof}
Let us show the inequality ``$\le$''. Let $u\in V_0^kL^{n/k,1}(\Omega)$.  Then
$u=u^+ - u^-$ where $u^+$ and $u^-$ are the positive and the negative part of $u$,
respectively. Fix $\varepsilon>0$.  Let $\psi_{\varepsilon}$ be the standard
mollification kernel supported in an open set $B_\varepsilon$ and let
$u_{\varepsilon}=\psi_{\varepsilon}*u$,
$u_{\varepsilon}^+=\psi_{\varepsilon}*u^+$ and
$u_{\varepsilon}^-=\psi_{\varepsilon}*u^-$.
Denote $\Omega_\varepsilon=\Omega+B_\varepsilon$.
Then both $u_{\varepsilon}^+$ and
$u_{\varepsilon}^-$ are supported in $\Omega_\varepsilon$ and
belong to $V_0^kL^{n/k,1}(\Omega_\varepsilon)$.
If we denote
\begin{equation} \label{E:emb-limiting-epsilon}
	I_\varepsilon\colon
		\dom[\Omega_\varepsilon]
			\to L^{\infty}(\Omega_\varepsilon),
\end{equation}
then, by shrinking property of \eqref{E:emb-limiting-epsilon} ensured by
Proposition~\ref{P:shrinking-emb-limiting}
\begin{equation} \label{E:Iepsilon-I}
	\nrm{I_\varepsilon} = \nrm{I},
\end{equation}
as $\Omega$ is an open subset of $\Omega_\varepsilon$.
Also, by linearity,
$u_{\varepsilon}=u_{\varepsilon}^+-u_{\varepsilon}^-$.  Note that, since $u$
vanishes at the boundary of $\Omega$,
\begin{equation*}
    \essinf u_{\varepsilon} \le 0 \le \esssup u_{\varepsilon}.
\end{equation*}
Moreover, one has
\begin{equation*}
	\esssup u_{\varepsilon}
		= \nrm{u_{\varepsilon}^+}_{L^\infty(\Omega)}
		\le \nrm{I_\varepsilon}\nrm{u_{\varepsilon}^+}_{\dom[\Omega_\varepsilon]}
		\le \nrm{I}\nrm{u_{\varepsilon}^+}_{\dom[\Omega_\varepsilon]}
\end{equation*}
and
\begin{equation*}
	- \essinf u_{\varepsilon}
		= \nrm{u_{\varepsilon}^-}_{L^\infty(\Omega)}
		\le \nrm{I_\varepsilon} \nrm{u_{\varepsilon}^-}_{\dom[\Omega_\varepsilon]}
		\le \nrm{I} \nrm{u_{\varepsilon}^-}_{\dom[\Omega_\varepsilon]},
\end{equation*}
where we used \eqref{E:Iepsilon-I}.
Summing both inequalities up, we get
\begin{equation} \label{E:sup-inf}
	\esssup u_{\varepsilon} - \essinf u_{\varepsilon}
		\le \nrm I \sum_{\abs{\multiindex}=k}
			\left(
				\nrm{D^\multiindex u_{\varepsilon}^+}_{L^{\nk,1}(\Omega_\varepsilon)}
				+ \nrm{D^\multiindex u_{\varepsilon}^-}_{L^{\nk,1}(\Omega_\varepsilon)}
			\right),
\end{equation}
as the definition of the Sobolev norm yields. Next, since $n/k\ge 1$, the
elementary inequality
\begin{equation} \label{E:elementary-ineq}
	a + b \le 2^{1-\kn} \left( a^\nk + b^\nk \right)^\kn
\end{equation}
holds for any non-negative $a$ and $b$. Therefore, applying \eqref{E:elementary-ineq} to
\begin{equation*}
	a = \nrm{D^\multiindex u_{\varepsilon}^+}_{L^{\nk,1}(\Omega_\varepsilon)}
		\quad\text{and}\quad
	b = \nrm{D^\multiindex u_{\varepsilon}^-}_{L^{\nk,1}(\Omega_\varepsilon)},
\end{equation*}
we obtain from~\eqref{E:sup-inf}
\begin{equation} \label{E:sup-inf-new}
	\esssup u_{\varepsilon} - \essinf u_{\varepsilon}
		\le 2^{1-\kn}\nrm I \sum_{\abs{\multiindex}=k}
			\left(
				\nrm{D^\multiindex u_{\varepsilon}^+}_{L^{\nk,1}(\Omega_\varepsilon)}^\nk
				+ \nrm{D^\multiindex u_{\varepsilon}^-}_{L^{\nk,1}(\Omega_\varepsilon)}^\nk
			\right)^\kn.
\end{equation}
By a particular case of~\cite[Proposition~2.5]{Bou:19}, the Lorentz space
$L^{\nk,1}(\Omega_\varepsilon)$ is disjointly superadditive with power
$n/k$, whence we get, for each multiindex $\multiindex$,
\begin{equation}\label{E:Dalpha-super-addition}
	\nrm{D^\multiindex u_{\varepsilon}^+}_{L^{\nk,1}(\Omega_\varepsilon)}^\nk
			+ \nrm{D^\multiindex u_{\varepsilon}^-}_{L^{\nk,1}(\Omega_\varepsilon)}^\nk
		\le \nrm{D^\multiindex u_{\varepsilon}}_{L^{\nk,1}(\Omega_\varepsilon)}^\nk.
\end{equation}
Altogether, \eqref{E:sup-inf-new} and \eqref{E:Dalpha-super-addition} yield
\begin{equation*}
	\esssup u_{\varepsilon} - \essinf u_{\varepsilon}
		\le  2^{1-\kn} \nrm I \nrm{u_{\varepsilon}}_{\dom[\Omega_\varepsilon]}.
\end{equation*}
The desired inequality now follows on letting $\varepsilon\to0_+$.

To show the converse inequality, let $\varepsilon>0$ be given
and suppose that $B_1$ and $B_2$ are two disjoint balls of
the same measure contained in $\Omega$ and let $u_1$ be a function
supported in $B_1$ such that $\nrm{u_1}_{L^\infty(\Omega)}\ge\nrm{I}-\varepsilon$
and $\nrm{u}_{{V_0^kL^{n/k,1}(\Omega)}}=1$.
The existence of such a function is guaranteed by shrinking property
of \eqref{E:emb-limiting} due to Proposition~\ref{P:shrinking-emb-limiting}.
Denote by $u_2$ the shift of $u_1$ onto
the domain $B_2$ and define $v\colon\Omega\to\R$ by
\begin{equation} \label{E:u1-u2}
	v = u_1\chi_{B_1} - u_2 \chi_{B_2}.
\end{equation}
Then $v\in\dom$ and
\begin{equation} \label{E:sup-inf-2}
	\esssup v - \essinf v
		\ge 2(\nrm{I} - \varepsilon).
\end{equation}
Next, observe that
\begin{equation*}
	\abs{D^\multiindex v}^*(t) =
	\begin{cases}
		\abs{D^\multiindex u_1}^*\left( \tfrac t2 \right)
			&\text{if $t\in(0,2\abs{B_1})$}
			\\
		0
			&\text{if $t\in(2\abs {B_1}, \abs\Omega)$}
	\end{cases}
\end{equation*}
for every multiindex $\multiindex$.
This implies
\begin{align} \label{E:norm-of-v}
	\begin{split}
	\nrm{v}_{\dom}
		& = \sum_{\abs\multiindex=k} \int_{0}^{2\abs {B_1}}
			\abs{D^\multiindex  v}^*(t)\, t^{\kn-1}\,\d t
			= \sum_{\abs\multiindex=k} \int_{0}^{2\abs {B_1}}
			\abs{D^\multiindex  u_1}^*\left(\tfrac t2\right)\, t^{\kn-1}\,\d t
			\\
		& = 2^\kn \sum_{\abs\multiindex=k} \int_{0}^{\abs {B_1}}
			\abs{D^\multiindex  u_1}^*(s)\, s^{\kn-1}\,\d s
			= 2^\kn \nrm{u_1}_{\dom}
			= 2^\kn
	\end{split}
\end{align}
and \eqref{E:sup-inf-2} together with \eqref{E:norm-of-v}
gives
\begin{equation*}
	\esssup v-\essinf v
		\ge 2^{1-\kn}(\nrm{I}-\varepsilon)\nrm{v}_{\dom}.
\end{equation*}
The desired lower bound for $\sigma(I)$ then follows by sending $\varepsilon \to 0_+$.
\end{proof}

Finally, we will show that, for~\eqref{E:emb-limiting}, there is equality
in~\eqref{E:I-rho}, that is,
\begin{equation*}
    \bmc(I)=\frac{\sigma(I)}{2}.
\end{equation*}
Combined with~\eqref{E:comparison-of-I-and-L}, this provides us with an exact evaluation of the measure of non-compactness of $I$ in~\eqref{E:emb-limiting}. In particular, this yields a negative answer to Question~\ref{Q:3}.

\begin{theorem} \label{T:beta-of-emb-limiting}
Let $n,k\in\N$, $k\le n$, let $\Omega\subset\rn$ be open bounded and nonempty set
and
let $I$ denote the embedding from \eqref{E:emb-limiting}.
Then
\begin{equation*}
    \bmc(I)=2^{-\kn}\nrm{I}.
\end{equation*}
In particular, $I$ is not maximally non-compact.
\end{theorem}

\begin{proof}
Due to Propositions~\ref{P:beta-upper-of-X-to-ell-infty}
and~\ref{P:span-of-emb-limiting}, we only need to
prove that
\begin{equation}\label{E:beta-lower}
    \bmc(I)\ge 2^{-\kn} \nrm{I}.
\end{equation}
Let $\rpar>0$ obey $2^{-k/n}\nrm{I}>\rpar$. We show that then necessarily
$\bmc(I)\ge\rpar$ and \eqref{E:beta-lower} follows.
Thanks to Theorem~\ref{T:beta-lower-of-X-to-ell-infty},
it suffices to find the set of eligible functions
satisfying assumptions \eqref{E:X-norm-difference} and \eqref{E:u-infty-norm-large}.
Let $\ell\in\N$ be given.
Denote by $B_1, B_2,\dots,B_{\ell}$ pairwise disjoint
balls of the same volume and all contained in $\Omega$.
By the shrinking property of embedding~\eqref{E:emb-limiting}
ensured by Proposition~\ref{P:shrinking-emb-limiting},
there is a function $v_1\colon B_1\to\R$ supported in $B_1$ such
that
\begin{equation*}
	\nrm{v_1}_{V^{k}_0L^{\nk,1}(\Omega)}
		= 1
		\quad\text{and}\quad
	\nrm{v_1}_{L^{\infty}(\Omega)}
		> \rpar 2^{\kn},
\end{equation*}
since $\nrm{I}>\rpar 2^{k/n}$.
For each $j=2,\dots,\ell$, let $v_j$ denote a shifted copy of $v_1$ supported
on $B_j$.
Then, due to \eqref{E:u1-u2} and \eqref{E:norm-of-v},
we have $\nrm{v_i-v_j}=2^{k/n}$ for distinct $i,j=1,\dots,\ell$,
whence the functions $u_j=2^{-k/n}v_j$, $j=1,\dots,\ell$,
have the required properties.
\end{proof}

In the case $n=k=1$ one has $\nrm{I}=\sigma(I)=1/2$. This is easy to observe using the
fundamental theorem of calculus, see \eg~\citep{Bou:19}. Note that this observation is consistent with \eqref{E:comparison-of-I-and-L}. For $n\ge 2$ and
$k=1$, the inequality ``$\le$'' in \eqref{E:comparison-of-I-and-L} was shown
in~\cite[Theorem~3.5(ii)]{CP:98}.

\paragraph{Acknowledgment}

We would like to thank David E.~Edmunds and Jan Mal\'y for stimulating
discussions and useful ideas.  We are grateful to the referee for their
careful reading of the manuscript and their valuable comments.

\paragraph{Funding}

This research was partly funded by Czech Science Foundation grant
P201-18-00580S.

\begin{small}

\end{small}

\end{document}